\documentclass[a4paper]{article}

\usepackage[all]{xy}
\usepackage[leqno]{amsmath}
\usepackage{amssymb}
\usepackage[only,mapsfrom,heavycircles,lightning]{stmaryrd}
\usepackage{latexsym}
\usepackage{graphicx}
\usepackage{graphbox}
\usepackage{mathrsfs}   
\usepackage{titlesec} 
\usepackage{paralist}
\usepackage{tocloft}
\usepackage{xcolor}
\usepackage{upgreek}
\usepackage{lmodern}
\usepackage[T1]{fontenc}
\usepackage[utf8]{inputenc}
\usepackage{calc}
\usepackage[normalem]{ulem}

\usepackage[hidelinks=true,hypertexnames=false]{hyperref}

\newlength{\theorempostskipamount}
\makeatletter

\newenvironment{theorem}[1][]
{\paragraph{Theorem} {\normalfont #1} \it}
{\vspace{\the\theorempostskipamount}}
\def\theorem{\@ifnextchar[{\@theoremopt}{\@theoremplain}}
\def\@theoremplain{\paragraph{Theorem} \it}
\def\@theoremopt[#1]{\paragraph{Theorem \normalfont #1}  \it}

\newenvironment{lemma}[1][]
{\paragraph{Lemma} {\normalfont #1} \it}
{\vspace{\the\theorempostskipamount}}
\def\lemma{\@ifnextchar[{\@lemmaopt}{\@lemmaplain}}
\def\@lemmaplain{\paragraph{Lemma} \it}
\def\@lemmaopt[#1]{\paragraph{Lemma \normalfont #1}  \it}

\newenvironment{proposition}[1][]
{\paragraph{Proposition} {\normalfont #1} \it}
{\vspace{\the\theorempostskipamount}}
\def\proposition{\@ifnextchar[{\@propositionopt}{\@propositionplain}}
\def\@propositionplain{\paragraph{Proposition} \it}
\def\@propositionopt[#1]{\paragraph{Proposition \normalfont #1}  \it}

\newenvironment{definition}[1][]
{\paragraph{Definition} {\normalfont #1} \it}
{\vspace{\the\theorempostskipamount}}
\def\definition{\@ifnextchar[{\@definitionopt}{\@definitionplain}}
\def\@definitionplain{\paragraph{Definition} \it}
\def\@definitionopt[#1]{\paragraph{Definition \normalfont #1}  \it}

\newenvironment{corollary}[1][]
{\paragraph{Corollary} {\normalfont #1} \it}
{\vspace{\the\theorempostskipamount}}
\def\corollary{\@ifnextchar[{\@corollaryopt}{\@corollaryplain}}
\def\@corollaryplain{\paragraph{Corollary} \it}
\def\@corollaryopt[#1]{\paragraph{Corollary \normalfont #1}  \it}

\newenvironment{question}[1][]
{\paragraph{Question} {\normalfont #1} \it}
{\vspace{\the\theorempostskipamount}}
\def\question{\@ifnextchar[{\@questionopt}{\@questionplain}}
\def\@questionplain{\paragraph{Question} \it}
\def\@questionopt[#1]{\paragraph{Question \normalfont #1}  \it}

\newenvironment{problem}[1][]
{\paragraph{Problem} {\normalfont #1} \it}
{\vspace{\the\problempostskipamount}}
\def\problem{\@ifnextchar[{\@problemopt}{\@problemplain}}
\def\@problemplain{\paragraph{Problem} \it}
\def\@problemopt[#1]{\paragraph{Problem \normalfont #1}  \it}

\newenvironment{conjecture}[1][]
{\paragraph{Conjecture} {\normalfont #1} \it}
{\vspace{\the\theorempostskipamount}}
\def\conjecture{\@ifnextchar[{\@conjectureopt}{\@conjectureplain}}
\def\@conjectureplain{\paragraph{Conjecture} \it}
\def\@conjectureopt[#1]{\paragraph{Conjecture \normalfont #1}  \it}

\newenvironment{remark}[1][]
{\paragraph{Remark} {\normalfont #1}}
{\vspace{\the\theorempostskipamount}}
\def\remark{\@ifnextchar[{\@remarkopt}{\@remarkplain}}
\def\@remarkplain{\paragraph{Remark}}
\def\@remarkopt[#1]{\paragraph{Remark \normalfont #1}}

\newenvironment{remarks}[1][]
{\paragraph{Remarks} {\normalfont #1}}
{\vspace{\the\theorempostskipamount}}
\def\remarks{\@ifnextchar[{\@remarksopt}{\@remarksplain}}
\def\@remarksplain{\paragraph{Remarks} \it}
\def\@remarksopt[#1]{\paragraph{Remarks \normalfont #1}  \it}

\newenvironment{example}[1][]
{\paragraph{Example} {\normalfont #1}}
{\vspace{\the\theorempostskipamount}}
\def\example{\@ifnextchar[{\@exampleopt}{\@exampleplain}}
\def\@exampleplain{\paragraph{Example}}
\def\@exampleopt[#1]{\paragraph{Example \normalfont #1}}

\newenvironment{examples}[1][]
{\paragraph{Examples} {\normalfont #1}}
{\vspace{\the\theorempostskipamount}}
\def\examples{\@ifnextchar[{\@examplesopt}{\@examplesplain}}
\def\@examplesplain{\paragraph{Examples} \it}
\def\@examplesopt[#1]{\paragraph{Examples \normalfont #1}  \it}

\newenvironment{exercise}[1][]
{\paragraph{Exercise} {\normalfont #1}}
{\vspace{\the\theorempostskipamount}}
\def\exercise{\@ifnextchar[{\@exerciseopt}{\@exerciseplain}}
\def\@exerciseplain{\paragraph{Exercise} \it}
\def\@exerciseopt[#1]{\paragraph{Exercise \normalfont #1}  \it}

\newenvironment{notation}[1][]
{\paragraph{Notation} {\normalfont #1}}
{\vspace{\the\theorempostskipamount}}
\def\notation{\@ifnextchar[{\@notationopt}{\@notationplain}}
\def\@notationplain{\paragraph{Notation} \it}
\def\@notationopt[#1]{\paragraph{Notation \normalfont #1}  \it}

\newenvironment{convention}[1][]
{\paragraph{Convention} {\normalfont #1}}
{\vspace{\the\theorempostskipamount}}
\def\convention{\@ifnextchar[{\@conventionopt}{\@conventionplain}}
\def\@conventionplain{\paragraph{Convention} \it}
\def\@conventionopt[#1]{\paragraph{Convention \normalfont #1}  \it}

\newenvironment{warning}[1][]
{\paragraph{Warning} {\normalfont #1}}
{\vspace{\the\theorempostskipamount}}
\def\warning{\@ifnextchar[{\@warningopt}{\@warningplain}}
\def\@warningplain{\paragraph{Warning} \it}
\def\@warningopt[#1]{\paragraph{Warning \normalfont #1}  \it}

\newenvironment{theoreme}[1][]
{\paragraph{Théorème} {\normalfont #1} \it}
{\vspace{\the\theorempostskipamount}}
\def\theoreme{\@ifnextchar[{\@theoremeopt}{\@theoremeplain}}
\def\@theoremeplain{\paragraph{Théorème} \it}
\def\@theoremeopt[#1]{\paragraph{Théorème {\normalfont #1}}  \it}

\newenvironment{lemme}[1][]
{\paragraph{Lemme} {\normalfont #1} \it}
{\vspace{\the\theorempostskipamount}}
\def\lemme{\@ifnextchar[{\@lemmeopt}{\@lemmeplain}}
\def\@lemmeplain{\paragraph{Lemme} \it}
\def\@lemmeopt[#1]{\paragraph{Lemme {\normalfont #1}}  \it}

\newenvironment{de'finition}[1][]
{\paragraph{Définition} {\normalfont #1} \it}
{\vspace{\the\theorempostskipamount}}

\newenvironment{exemple}[1][]
{\paragraph{Exemple} {\normalfont #1}}
{\vspace{\the\theorempostskipamount}}
\def\exemple{\@ifnextchar[{\@exempleopt}{\@exempleplain}}
\def\@exempleplain{\paragraph{Exemple}}
\def\@exempleopt[#1]{\paragraph{Exemple {\normalfont #1}}}

\newcommand{\thmendspace}{\vspace{\the\theorempostskipamount}}
\newlength{\negvcorr}
\newlength{\mnegvcorr}
\makeatother

\newenvironment{proof}[1][Proof]{\bigskip\noindent \textit{#1.~}}
{\hfill $\Box$}

\makeatletter

\def\@removefromreset#1#2{\let\@tempb\@elt
   \def\@tempa#1{@&#1}\expandafter\let\csname @*#1*\endcsname\@tempa
   \def\@elt##1{\expandafter\ifx\csname @*##1*\endcsname\@tempa\else
         \noexpand\@elt{##1}\fi}%
   \expandafter\edef\csname cl@#2\endcsname{\csname cl@#2\endcsname}%
   \let\@elt\@tempb
   \expandafter\let\csname @*#1*\endcsname\@undefined}

\@addtoreset{paragraph}{section}
\@addtoreset{paragraph}{part}
\@removefromreset{paragraph}{subsubsection}
\@removefromreset{paragraph}{subsection}

\let\c@equation\c@subparagraph

\makeatother

\renewcommand{\theparagraph}{(\arabic{section}.\arabic{paragraph})}
\renewcommand{\thesubparagraph}
{(\arabic{section}.\arabic{paragraph}.\arabic{subparagraph})}

\titleformat{\part}[display]{\normalfont\Large\bfseries}%
{\partname}{0cm}{}

\titleformat{\section}[hang]{\normalfont\Large\bfseries}{}{0cm}%
{\thesection \  --\ }

\titleformat{\subsection}[hang]{\normalfont\large\bfseries}{}{0cm}%
{\thesubsection \  --\ }

\newcommand{\spcifnec}[1]
{\ifx#1\empty
\else ~#1.
\fi}

\titleformat{\paragraph}[runin]{\normalfont\bfseries}
{\theparagraph}{0cm}{\spcifnec}
\titlespacing{\paragraph}{0cm}
{2.75ex plus 1ex minus .2ex}
{.5em}

\titleformat{\subparagraph}[runin]{\it}
{\thesubparagraph}{0cm}{\spcifnec}
\titlespacing{\subparagraph}{0cm}
{0mm}
{.5em}

\makeatletter
\let\coresubpar\subparagraph
\def\subparagraph{\@ifnextchar[{\@varsubpar}{\coresubpar}}
\def\@varsubpar[]#1{\coresubpar{}{\it \ \kern -.45em #1}}
\let\intersubpar\subparagraph
\def\subparagraph{\@ifnextchar*{\@starredsubpar}{\intersubpar}}
\def\@starredsubpar*{\@ifnextchar[{\@varstarredsubpar}{\@plainstarredsubpar}}
\def\@varstarredsubpar[]#1{\par\noindent {\it #1}}
\def\@plainstarredsubpar#1{\par\noindent {\it #1.}}

\let\corepar\paragraph
\def\paragraph{\@ifnextchar[{\@varpar}{\corepar}}
\def\@varpar[]#1{\corepar{}{\bf \ \kern -.45em #1}}
\let\interpar\paragraph
\def\paragraph{\@ifnextchar*{\@starredpar}{\interpar}}
\def\@starredpar*{\@ifnextchar[{\@varstarredpar}{\@plainstarredpar}}
\def\@varstarredpar[]#1{\bigskip\par\noindent {\bf #1}}
\def\@plainstarredpar#1{\bigskip\par\noindent {\bf #1.}}

\makeatother

{\titleformat{\section}[hang]{\normalfont\large\bfseries}{}{0cm}{}
\titleformat{\subsection}[hang]{\normalfont\bfseries}{}{0cm}{}
\renewcommand{\theparagraph}{(\Alph{paragraph})}

\maketitle
}{}

\newenvironment{closing}%
{\titleformat{\section}[hang]{\normalfont\large\bfseries}{}{0cm}{}
\setlength{\itemsep}{0mm}
\small
}
{}

\makeatletter
\renewcommand\@maketitle{%
  \newpage
  \begin{center}%
  \let \footnote \thanks
    {\Large \bf \@title \par}%
    \vskip 1em%
    {\large
      \begin{tabular}[t]{c}%
        \@author
      \end{tabular}\par}%
  \end{center}%
  \par
  \vskip 1.5em}
\makeatother

\renewenvironment{abstract}
{\small \quotation
\noindent {\bf Abstract.}}{\endquotation \vskip 1cm}

\newcommand{\C}{\mathbf{C}}

\newcommand{\Z}{\mathbf{Z}}
\newcommand{\N}{\mathbf{N}}
\renewcommand{\P}{\mathbf{P}}

\newcommand{\PH}{\mathbf{P}\kern -.05em \mathrm{H}}

\newcommand{\F}{\mathbf{F}}

\renewcommand{\O}{\mathcal{O}}
\newcommand{\I}{\mathcal{I}}

\renewcommand{\k}{\mathbf{k}}

\newcommand{\M}{\mathcal{M}}
\newcommand{\K}{\mathcal{K}}

\newcommand{\Cliff}{\mathrm{Cliff}}

\newcommand{\Ext}{\mathrm{Ext}}

\newcommand{\Pic}{\mathrm{Pic}}

\newcommand{\red}{\mathrm{red}} 
\newcommand{\PGL}{\mathrm{PGL}}

\newcommand{\cHom}{{\cal H}\kern -.08em om} 
\newcommand{\cExt}{{\cal E}\kern -.1em xt} 

\newcommand{\vect}[1]{\langle #1 \rangle} 

\DeclareMathOperator{\cork}{cork}

\DeclareMathOperator{\coker}{coker}
\DeclareMathOperator{\expdim}{expdim}

\newcommand{\dlbrack}{[ \kern -.4ex [}
\newcommand{\drbrack}{] \kern -.4ex ]}

\newcommand{\trsp}[1]{\vphantom{#1}^{\mathsf T\!} #1}
\newcommand{\restr}[2]{\left. #1 \right| _{#2}}

\makeatletter
\def\@orthpar[#1]{(#1)^\perp}
\def\@orthst#1{#1^\perp}
\def\orth{\@ifnextchar[{\@orthpar}{\@orthst}}
\makeatletter

\makeatletter
\def\@dualpar[#1]{(#1)^\vee}
\def\@dualst#1{#1^\vee}
\def\dual{\@ifnextchar[{\@dualpar}{\@dualst}}
\makeatletter

\newcommand{\bx}{\mathbf x}

\renewcommand{\epsilon}{\varepsilon}
\renewcommand{\geq}{\geqslant}
\renewcommand{\leq}{\leqslant}

\def\subset{\subseteq}
\renewcommand{\emptyset}{\varnothing}

\newcommand{\ie}{i.e.,\ } 
\newcommand{\cf}{cf.\ } 
\newcommand{\eg}{e.g.,\ }
\newcommand{\noeud}{n{\oe}ud}
\newcommand{\noeuds}{n{\oe}uds}
\makeatletter
\def\noeud{\@ifnextchar.{n{\oe}ud}{\@ifnextchar,{n{\oe}ud}{n{\oe}ud\ }}}
\def\noeuds{\@ifnextchar.{n{\oe}uds}{\@ifnextchar,{n{\oe}uds}{n{\oe}uds\ }}}
\makeatother
\def\?{?\kern -.08em ?}
\def\wtf{?\kern -.08em !}

\newcommand{\KC}{\mathcal{KC}}

\renewcommand{\F}{\mathcal{F}}
\newcommand{\FS}{\mathcal{FS}}

\let\Zint\Z 

\def\VW#1{\mathcal{V\kern -.1em W}_{#1}}
\def\V#1{\mathcal{V}_{#1}}
\let\ZZ\Z
\def\Z#1{\mathcal{Z}_{#1}}
\def\W#1{\mathcal{W}_{#1}}

\def\sKC{K\kern -.15em C}
\let\CC\C
\def\C#1{\mathcal{C}_{#1}}
\let\SS\S
\def\S{\mathcal{S}}

\title{Double covers and extensions}
\author{Ciro Ciliberto and Thomas Dedieu}

\begin{document}

\setdefaultenum{(i)}{}{}{}
\setdefaultitem{---}{}{}{}

\maketitle

\begin{abstract}
In this paper we consider double covers of the projective space in
relation with the problem of extensions of varieties, specifically of
extensions of canonical curves to $K3$ surfaces and Fano 3-folds. In
particular we consider $K3$ surfaces which are double covers of the
plane branched over a general sextic: we prove that the general curve
in the linear system pull back of plane curves of degree $k\geq 7$
lies on a unique $K3$ surface. If $k\leq 6$ the general such curve is
instead extendable to a higher dimensional variety. In the cases
$k=4,5,6$, this gives the existence of singular index $k$ Fano
varieties of dimensions 8, 5, 3, genera 17, 26, 37,
and indices 6, 3, 1
respectively. For $k = 6$ we recover the Fano variety $\P(3, 1, 1,
1)$, one of only two Fano threefolds with canonical Gorenstein
singularities with the maximal genus 37, found by Prokhorov.
We show that the latter variety is no further extendable.
For $k=4$ and $5$ these Fano varieties have been identified by Totaro.
We also study the
extensions of smooth degree 2 sections of $K3$ surfaces of genus 3.
In all these cases, we compute the co-rank of the Gauss--Wahl maps of
the curves under consideration.  Finally we observe that linear
systems on double covers of the projective plane provide superabundant
logarithmic Severi varieties.
\end{abstract}


Let $\pi: V \to \P^n$ be the double cover branched over a smooth
degree $2d$ hypersurface $B$, and $L$ be the pull-back to $V$ of the
hyperplane class. For $k\geq d$, the general membre $W$ of $|kL|$ is
the normalization of a degree $2k$ hypersurface $W^\flat$ in $\P^n$
everywhere 
tangent to $B$, with an ordinary double singularity along a smooth
complete intersection $Z$ of type $(k,k-d)$, and such that there
exists a degree $k$ hypersurface containing $Z$ and cutting out on $B$
its contact locus with $W^\flat$.
Our main observation in this article is the quite surprising fact that
given $W^\flat$ with an ordinary double singularity along the smooth
complete intersection $Z$, there always exists $B$ fulfilling all the
other requirements,
so that $W^\flat$ lifts to the corresponding double cover of $\P^n$
(see Section~\ref{S:main} and in particular
Proposition~\ref{pr:everywhere-tg}).

This gives us the possibility to describe smooth
degree $k$ sections of $K3$ surfaces of genus $2$
(sextic double planes)
and to study their extensions.
In particular we compute the coranks of their Gauss--Wahl
maps 
(see subsection~\ref{S:extensions} for the relevant definitions).
More precisely,
let $\K_2$ be the moduli stack of primitively polarized $K3$ surfaces
$(S,L)$ of 
genus $2$, and $\KC _{k^2+1}^k$ the moduli stack of triples $(S,L,C)$
such that $[S,L] \in \K_2$ and $C$ is a smooth membre of the linear
system $|kL|$.
One has the forgetful map
$c_{k^2+1} ^k: [S,L,C] \in \KC _{k^2+1}^k \mapsto [C] \in
\M_{k^2+1}$.
We determine the dimension of the general fibre of $c_{k^2+1} ^k$, and when
$k>3$ this gives us $\cork(\Phi_C)$
(the co-rank of the Gauss--Wahl map of $C$,
see subsection~\ref{S:extensions}) 
for the general $C$ in the image
of $c_{k^2+1}^k$ by \cite[Thm.~2.6]{cds}; when $k=3$ (resp.\ $k\leq 2$),
$C$ is a plane sextic (resp.\ hyperelliptic) and the relevant
cohomological invariants were 
already known. This is stated in Theorem~\ref{t:g2}.
The coranks of the Gauss--Wahl maps had been found previously in
\cite{CLM98}, but the values given in \cite[Table~2.14]{CLM98} are
wrong for $k=4,5,6$,
as has first been pointed out by Totaro \cite[Ex.~5.2]{totaro0};
see the corrigendum to \cite{CLM98} for details.
The approach in \cite{CLM98} is completely different, purely
cohomological. 

We find that for $k \geq 7$ the general curve in the image of $c_{k^2+1}^k$
lies on a 
unique $K3$ surface, hence it is extendable only one step;
this may also be seen from the results in \cite{CLM98}.
For $k\leq 6$ however, we find that $c_{k^2+1}^k$ has positive dimensional
fibres, hence the general curve in its image is extendable to a higher
dimensional variety  by \cite{cds}. For $k=4,5,6$ respectively, this
gives the 
existence of singular Fano varieties of dimensions $8,5,3$, genera
$17,26,37$ and indices $4,5,6$ respectively.
For $k=6$ this turns out to be $\P(3,1,1,1)$, one of only two Fano
threefolds with canonical Gorenstein singularities with the maximal
genus $37$, as has been proved by Prokhorov \cite{prokhorov}.
Our results show that $\P(3,1,1,1)$ is not extendable.
For $k=4$ and $5$, the Fano varieties have been identified by Totaro
as sextic hypersurfaces in weighted projective spaces, see
\ref{p:totaro}.

As another application of our main observation, we study
in Section~\ref{S:double-Fano} the
extensions of smooth degree $2$ sections of $K3$ surfaces of genus
$3$. Indeed a general $K3$ surface of genus $3$ is a smooth quartic
$S$, hence it may be realized as an anticanonical divisor of various
double covers $V$ of $\P^3$ branched over a quartic; the linear
curve sections of $V$ in its anticanonical embedding are complete
intersections of type $(2,4)$ in $\P^3$.
In this case the results of \cite{cds} do not apply because the curves
under consideration have Clifford index $2$, however we compute the
relevant cohomological invariants by hand (Proposition~\ref{pr:cohom})
and observe that they fit with our description of the extensions.

Our main result in this article is also applied in \cite{th-5iche} to
analyze the extensions of plane quintics in their canonical model.
In the follow-up article \cite{cd-highind} we give a systematic
description of all the maps $c_g^k: \KC _g ^k \to \M_k$ with $k>1$
that have positive-dimensional general fibre (see the notation in
subsection~\ref{S:extensions}).

In Section~\ref{S:severi} we make an observation of a different nature,
namely that linear systems on double covers of the projective plane
provide superabundant logarithmic Severi varieties. The actual
dimension exceeds the expected dimension by the geometric genus of the
double cover. We guess that there should be a conceptual explanation
of this fact, but couldn't find it.

\bigskip
\noindent
\textbf{Thanks.}
This article grew out of a suggestion of Edoardo Sernesi that there
should be a relation between Prokhorov's extremal Fanos
\cite{prokhorov} and the Donagi--Morrison example
\cite{donagi-morrison}; we thank him for his constant help and
interest. ThD benefited from conversations with Justin
Sawon which have been of great importance in the development of this
project.
We are grateful to Burt Totaro for sharing with us the Fano varieties
that he identified as the universal extensions of $K3$ double sextic
planes in the cases $k=4$ and $5$.
We also thank Jason Starr, who spotted a flaw in the first version of
this article, in the proof of our main theorem.
Eventually we thank the referee for his careful reading of this text.


\section{Notation and background}

\subsection{General notation and convention}

We work over the field of complex numbers.

We call $H$ the hyperplane
class of the projective space $\P^n$, and use $h(d)$ as a shorthand
for $h^0(\P^n, dH)$.

For $A$ and $B$
two linearly equivalent, distinct and effective Cartier divisors
on a projective variety $X$, we
denote by $\vect {A,B}$ the pencil generated by $A$ and $B$.

\subsection{Extensions and ribbons}
\label{S:extensions}

We use freely troughout the text the notions of extensions, ribbons,
etc. They are carefully treated in \cite{cds}, but we include here a
short reminder for the reader's convenience.
The material in this subsection is not involved in the proof of our
main result Theorem~\ref{t:main}; we use it for the applications in
Sections \ref{S:K3-g2} and \ref{S:double-Fano}.

\paragraph{}
A projective variety $X \subset \P^n$ is \emph{extendable} $r$ steps
if there exists a projective variety $Y \subset \P^{n+r}$, not a cone,
and having $X$ as a linear section. The variety $Y$ is then called an
($r$ steps) \emph{extension} of $X$.

It has been proved by Lvovski \cite{lvovskiII} that the extendability
of $X \subset \P^n$ is controlled by the invariant
\[
\alpha(X) = h^0(N_{X/\P^n}(-1))-n-1,
\]
namely if $X$ is smooth and irreducible, not contained in a
hyperplane, and not a quadric, and if $\alpha(X)<n$, then $X$ is
extendable at most $\alpha(X)$ steps.

When $X$ is a canonical curve $C \subset \P^{g-1}$ of genus $g$,
one has
\[
  \alpha(C) = \cork (\Phi_C),
\]
the co-rank of the Gauss--Wahl map
$\Phi_C : \bigwedge^2 H^0(C,K_C) \to H^0(C,3K_C)$ which is defined by
linearity and the relations $s \wedge t \mapsto s\cdot dt - t \cdot
ds$.
When $X$ is a linearly normal $K3$ surface $S \subset \P^g$, one has
\[
  \alpha(S) = h^1(T_S(-1)).
\]

\paragraph{}
If $X\subset \P^n$ is either a canonical curve $C \subset \P^{g-1}$ or
a linearly normal $K3$ surface $S \subset \P^g$ of genus $g \geq 11$ and
Clifford index $>2$, then there holds the following strong converse to
Lvovski's Theorem, see \cite{abs1} and \cite{cds}: there exists $Y
\subset \P^{n+\alpha(X)}$ an $\alpha(X)$ steps extension of $X$, such
that every $1$ step extension $X' \subset \P^{n+1}$ (up to
projectivities of $\P^{n+1}$ leaving $X$ fixed) appears in a
unique way as a linear section of $Y$ containing $X$. 
We call $Y$ the \emph{universal extension} of $X$.

In particular, under the above assumptions the $1$ step extensions of
$X \subset \P^n$ fit in a universal family parametrized by a
projective space of 
dimension $\alpha(X)-1$.

\paragraph{}
Let us denote by:\\
\begin{inparaitem}
\item $\M_g$ the moduli stack of smooth curves of genus $g$;\\
\item  $\K_g^k$ the moduli stack of polarised $K3$ surfaces $(S,L)$ of
  genus $g$ such that $L$ has divisibility exactly $k$,
\ie $S$ is a smooth $K3$ surface, and $L$ is an ample,
globally generated
line bundle on $S$ with $L^ 2=2g-2$, such that $L=kL'$ with $L'$ a
primitive line bundle; \\
\item $\KC_g^k$
the moduli stack of pairs $(S,C)$
with $C$ a smooth curve on $S$ and $[S,\O_S(C)]\in \K_g^k$;\\
\item $c_g^k: \KC_g^k\to \M_g$ the forgetful map
$[S,C] \mapsto [C]$.
\end{inparaitem}
\\
If $g\geq 11$ and $[S,C] \in \KC_g^k$ is such that
$\Cliff(C)>2$, one has
\[
\dim \bigl( (c_g^k)^{-1}(C) \bigr) = \cork(\Phi_C)-1,
\]
see \cite[Thm.~2.6]{cds}.
We shall also consider:\\
\begin{inparaitem}
\item $\F_g ^k$ the moduli stack of Fano threefolds $V$ of genus $g$
  and index $k$, \ie
smooth varieties $V$ with $-K_V$ ample equal to $kL'$ for some
primitive line bundle $L'$, and $K_V^ 3=2-2g$;\\
\item $\FS_g^k$ the moduli stack of pairs $(V,S)$ with $V\in \F_g^k$ and
$S\in |-K_V|$ a smooth surface;\\
\item $s_g^k: \FS_g^k\to \K_g$ the
forgetful map $[V,S] \mapsto [S]$ (with $\K_g$ the union of all
$\K_g^k$, $k>0$).
\end{inparaitem}
\\
Again it holds that if $g\geq 11$ and $\Cliff(S)>2$ then the fibre of
$s_g^k$ over $S$ has dimension $h^1(T_S(-1))-1$
\cite[Thm.~2.19]{cds}, but we will not use
this in this text.

\paragraph{}
A \emph{ribbon} over $X\subset \P^n$ (a projective variety, as always)
with normal bundle $\O_X(1)$ is a scheme 
$\tilde X$ such that $\tilde X _\red =X$,
$\I_{X/\tilde{X}}^2=0$,
and $\I_{X/\tilde{X}} = \O_X(-1)$.
These ribbons are uniquely determined by their extension classes
$e _{\tilde X} \in \Ext^1(\Omega^1_X,\O_X(-1))$, and two ribbons are
isomorphic if and only if their extension classes are proportional.

If $X$ is smooth and $Y \subset \P^{n+1}$ is an extension of $X$, then
the first infinitesimal neighbourhood of $X$ in $Y$ is a ribbon over
$X$ with normal bundle $\O_X(1)$, which we denote by $2X_Y$.
In this case we say that $Y$ is an \emph{integral} of the ribbon
$2X_Y$. The extension class of the ribbon $2X_Y$ lies in the kernel of
the map
\[
\eta: \Ext^1(\Omega^1_X,\O_X(-1)) \to 
\Ext^1(\restr {\Omega^1_{\P^n}} X,\O_X(-1))
\]
induced by the conormal exact sequence of $X \subset \P^n$
\cite[Lem.~4.4]{cds}.

When $X \subset \P^n$ is a canonical curve $C \subset \P^{g-1}$
(resp.\ a linearly normal $K3$ surface $S \subset \P^g$) the map
$\eta$ is $\trsp \Phi_C$ (resp.\ $0$).
It follows that $\coker(\Phi_C)$ (resp.\ $H^1(T_S(-1))$) parametrizes
those ribbons likely to be integrated to an extension of $X$.
The central point in \cite{abs1} and \cite{cds} is that when
$g\geq 11$ and $\Cliff>2$, each such ribbon may be integrated in a unique
way to an extension of $X$ (up to projectivities leaving $X$ fixed).

\paragraph{}
Under suitable assumptions, and in particular if $X \subset \P^n$ is
either a canonical curve or a linearly normal $K3$ surface, one may
interpret $H^0(N_{X/\P^n}(-1))$ as parametrizing all possible
embeddings in $\P^{n+1}$ of those ribbons over $X$ with extension
class in $\ker(\eta)$ (see \cite[\SS 3]{cds}). Note that two such
embeddings of the same ribbon differ by a projectivity of $\P^{n+1}$
leaving each point of $X$ fixed; these projectivities form the group
of homotheties and translations of the $n$-dimensional affine space,
and amount for the difference between 
$h^0(N_{X/\P^n}(-1))$ and $\alpha(X)$.

In general it is not true that the integral of a ribbon is unique, and
indeed we shall see examples of this in the present text, which
already appeared in \cite{beauville-merindol}.
We expect that in this case (under suitable assumptions) the
extensions of $X$ with fixed first infinitesimal neighbourhood are
controlled by
$\bigoplus _{l\geq 2} H^0(N_{X/\P^n}(-l))$.
More generally we expect that given a $k$-th order ribbon $\tilde X$
(where we take the order of an ordinary ribbon to be $2$)
over $X$ with normal bundle $\O_X(1)$ embedded in $\P^{n+1}$,
the Hilbert scheme of extensions of $X$ containing
$\tilde X$, if nonempty, is of dimension
$\sum _{l\geq k} h^0(N_{X/\P^n}(-l))$.

These expectations are verified in the situations we consider in this
article, see Theorems \ref{t:g2} and \ref{t:double-fano}.

\section{Divisors on double covers of the projective space}
\label{S:double-covers}

\paragraph{}
\label{p:double-covers}
We begin by reviewing some elementary facts about double covers.
Let $d$ be a positive integer. We consider a degree $2d$ hypersurface
$B$ in $\P^n$, and $\pi: V \to \P^n$ the double cover branched over
$B$. 
Let $H$ be the hyperplane class on
$\P^n$, and $L$ be its pull-back to $V$.
For all $k\in \N$ we have
\begin{equation}\label{eq:decomp:kL}
H^0(V,kL) = 
\pi^* H^0(\P^n, kH) \oplus \pi^* H^0(\P^n, kH-{\textstyle\frac 1
  2}B),
\end{equation}
which is the isotypic decomposition of 
$H^0(V,kL)$ as a representation of $\ZZ/2$.
The first summand corresponds to divisors that are double covers of
degree $k$ hypersurfaces in $\P^n$, and the second to divisors that
decompose as $B$ (seen as the ramification divisor in $V$) plus the
double cover of a degree $k-d$ hypersurface in $\P^n$.

\begin{proposition}\label{p:descr}
For $k\geq d$, the general member $W$ of $|kL|$ is not a double cover
of some hypersurface in $\P^n$, the restriction $\restr \pi W$ is
birational on 
its image, a degree $2k$ hypersurface $W^\flat$ in $\P^n$ everywhere
tangent to $B$, with an ordinary double singularity along a smooth
complete intersection $Z$ of type $(k,k-d)$, and such that there
exists a degree $k$ hypersurface containing $Z$ and cutting out on $B$
its contact locus with $W^\flat$.
\end{proposition}

\begin{proof}
The divisor $W$ belongs to a unique pencil $\vect {A', B+D'}$,
with $A'$ and $D'$ the double covers of hypersurfaces $A$ and $D$ in
$\P^n$ of respective degrees $k$ and $k-d$.
Thus $W^\flat :=\pi(W)$ belongs to the pencil 
$\vect {2A, B+2D}$, from which it follows that $W^\flat$ is double
along $Z:=A \cap D$, and touches $B$ doubly along $A\cap B$, which
accounts for the whole intersection scheme of $W^\flat$ and $B$. 
The base locus of this pencil is the scheme defined by the ideal 
$\I_Z^2(\I_A^2+\I_B)$.

The pull-back $\pi^* W^\flat \in |2kL|$ splits as $W+i(W)$, with $i$ the
involution on $V$ associated to $\pi$~; it has a double singularity along
$Z':= \pi^{-1}(Z)$ and $\pi^{-1}(B\cap A)$, with at each point
one local sheet belonging to $W$ and another to $i(W)$. The union 
$Z' \cup \pi^{-1}(B\cap A)$ is the base locus of the pencil 
$\vect {A', B+D'}$.
\end{proof}

\begin{remark}
Note that $\pi$ induces a $2:1$ map
$\vect {A', B+D'} \to \vect {2A, B+2D}$ wich ramifies exactly at the
two points $[A']$ and $[B+D']$. 
It is the restriction of the map $\pi_*:|kL|\to |2kH|$, which is $2:1$
on its image, the join of the two loci of divisors $2A$ and $B+2D$
respectively. 
Both these loci are to be regarded as $2$-Veronese varieties,
of respective degrees $2^{h(k)-1}$ and $2^{h(k-d)-1}$.
By the general point of the join there passes a unique joining line, so 
$\pi_*(|kL|)$  has
degree $2^{h(k)+h(k-d)-2}=2^{\dim (|kL|)-1}$ as a subvariety of
$|2kH|$ (see, \eg \cite[p.~236]{harris}).
\end{remark}

\paragraph{}\label{p:wps}
All the above may be conveniently seen by writing down equations.
The double cover $V$ is the degree $2d$ hypersurface
\(  y^2 = g_{2d}(\bx),\)
in the weighted projective space $\P(1^{n+1},d)$ with homogeneous
coordinates $(\bx,y)$, $\bx=(x_0,\ldots,x_n)$, where $g_{2d}$ is the
equation of $B$ in $\P^n$.

The linear system $|kL|$ is cut out on $V$ by weighted $k$-ics in
$\P(1^{n+1},d)$. For all $W \in |kL|$ there exist
homogeneous polynomials $f_{k-d}(\bx)$ and $f_k(\bx)$ of respective
degrees $k-d$ and $k$ such that the homogeneous ideal of $W$ in
$\P(1^{n+1},d)$ is
\begin{equation*}
  I_W = 
  \bigl(
  y^2 - g_{2d}(\bx),\
  y f_{k-d}(\bx) - f_k(\bx)
  \bigr).
\end{equation*}
In the decomposition \eqref{eq:decomp:kL} of
$H^0(V,kL)$, the first and second summands respectively consist of
polynomials of the form $f_k(\bx)$ and $y f_{k-d}(\bx)$, and the
involution $i$ is given by $(\bx,y) \mapsto (\bx,-y)$.
The divisors $A'$ and $B+D'$ are defined in $V$ by the equations
$f_k(\bx)=0$ and $y f_{k-d}(\bx)=0$ respectively.
Eliminating $y$ from the ideal $I_W$,
one finds the equations of $W^\flat$
in $\P^n$, namely
\begin{equation*}
  f_k (\bx)^2 = g_{2d}(\bx) f_{k-d}(\bx)^2.
\end{equation*}
Then the hypersurfaces $A$ and $D$ in $\P^n$ are defined by
$f_k(\bx)=0$ and $f_{k-d}(\bx)=0$ respectively.

\section{Families of double covers containing a divisor}
\label{S:main}

\paragraph{}
\label{p:notation:VW}
We shall consider double covers as in \ref{p:double-covers} and the
corresponding linear systems in families.
For all $d\in \N^*$ we let $\V d$ be the linear system of
degree $2d$ hypersurfaces in $\P^n$, which we consider as the
parameter space of double covers of $\P^n$ branched over a
$2d$-ic. 
Then for all $k \geq d$, we consider $\VW {k,d}$ the parameter space
for pairs $(V,W)$ with $V \in \V d$ and $W \in |kL|$ on $V$
(where as usual $L=\pi^*H$ on $V$, with $\pi:V \to \P^n$ the
double cover).

On the other hand for all $k \geq d$ we let 
$\Z {k,d}$ be the parameter space for complete intersections
of bidegree $(k,k-d)$ in $\P^n$
(with the convention that for $k=d$ this is just one point
corresponding to $Z=\emptyset$).
Eventually, we let $\W {k,d}$ be the parameter space for degree
$2k$ hypersurfaces in $\P^n$ singular along some $Z \in \Z {k,d}$.

From \ref{p:double-covers} we have a commutative diagram
\[
\xymatrix{
\VW {k,d} 
\ar[d]_p
\ar[dr]^q
\\
\V d
& \W {k,d}
}
\]
where $p$ and $q$ are the forgetful maps $(V,W)\mapsto V$ and
$(V,W)\mapsto W^\flat$ respectively.
Our main result is the following.

\begin{theorem}
\label{t:main}
The forgetful map
$q: \VW {k,d} \to \W {k,d}$
is dominant, with generic fibre of dimension
$h^0(\P^n,2k-d)$.
\end{theorem}

\bigskip
The following lemma is the keystone of the proof.
Recall that if $X \subset \P^n$ is a variety defined by the prime
ideal $I_X$, then for all positive integer $k$ the $k$-th symmetric
power $I_Z^{(k)}$ is the ideal of those polynomials vanishing at the
order at least $k$ at every point of $X$ (see, e.g.,
\cite[3.9]{eisenbud}).

\begin{lemma}
\label{l:resol}
Let $Z$ be the complete intersection in $\P^n$ of two hypersurfaces of
degrees $a$ and $b$ respectively.
The second symmetric power $I_Z^{(2)}$ of
the homogeneous ideal $I_Z \subset R=\k[x_0,\ldots,x_n]$
has the following resolution,
\begin{equation*}
0 \to 
R(-2a-b)\oplus R(-a-2b)
\to
R(-2a)\oplus R(-a-b)\oplus R(-2b)
\to
{I}_Z^{(2)} \to 0.
\end{equation*}
\end{lemma}

\vspace{\negvcorr}
\begin{proof}
By \cite[Appendix~6, Lemma~5]{zariski-samuel}, the
second symmetric power $I_Z^{(2)}$ equals
the square $I_Z^2$. We compute a resolution of $I_Z^2$ as follows.
The ideal $(x,y)^2 \subset \k[x,y]=S$ has the resolution
\begin{equation*}
0  \to
S(-3)^{\oplus 2}
\xrightarrow 
{\scriptsize
  \begin{pmatrix}
-y & 0 \\
x & -y \\
0 & x 
  \end{pmatrix}
}
S(-2)^{\oplus 3}
\xrightarrow 
{\scriptsize
(x^2, xy, y^2)
}
S \to 0,
\end{equation*}
as may be computed for instance using the method of
\cite[Cor.2.4]{eisenbud-syz}, and then removing the trivial extraneous
factors.
The resolution for $I_Z^2$ may then be obtained by replacing $x$ and
$y$ by two generators of $I_Z$, of degrees $a$ and $b$ respectively.
Alternatively, see \cite{guardo-vantuyl}.
\end{proof}

\begin{proposition}
\label{pr:everywhere-tg}
Let $k \geq d$ be two positive integers.
Let $W^\flat$ be a degree $2k$ hypersurface in $\P^n$,
with an ordinary double singularity along $Z$, a smooth 
complete intersection of type $(k,k-d)$ in $\P^n$, 
with the convention that $Z=\emptyset$ if $k=d$.
Assume moreover that $W^\flat$ is general among such hypersurfaces.
Then there exists a family 
of dimension at least $h^0(\P^n,\O(2d-k))$
of degree $2d$ hypersurfaces $B$ everywhere
tangent to $W^\flat$
and such that there exists a unique degree $k$ hypersurface
containing $Z$ and cutting out the contact locus of $B$ and $W^\flat$
on $B$.
\end{proposition}

\begin{proof}
Let us first prove the existence of one smooth $(2d)$-ic 
$B$ with the required properties.
Let $f_{k-d},f_k$ be generators of degrees $k-d$ and $k$ respectively
of the homogeneous ideal $I_Z$.
By the resolution of $I_Z^{(2)}$ given in Lemma~\ref{l:resol}, the
equation $g$ of $W^\flat$ may be written in the form
\begin{equation}\label{exprg1}
  g =
  f_{k-d}^2 g_{2d}+2 f_{k-d}f_k g_d + f_k^2,
\end{equation}
with $g_{2d},g_d$ homogeneous polynomials of degrees $2d,d$
respectively.

We now put this equation in an appropriate form:
\begin{align}
  \notag
  g=
  f_{k-d}^2 g_{2d}+2 f_{k-d}f_k g_d + f_k^2
 &=
   (f_k + f_{k-d} g_d)^2 + f_{k-d}^2 ( g_{2d} -g_d^2 )
  \\
\label{canform}
  &= \tilde f_k ^2 + f_{k-d}^2 \tilde g_{2d},
\end{align}
where $\tilde f_k = f_k + f_{k-d} g_d$, and
$\tilde g_{2d} = g_{2d} -g_d^2$.
Note that $\tilde f_k$ generates $I_Z$ together with
$f_{k-d}$.
By generality of $W^\flat$ the  $(2d)$-ic $B$ defined by $\tilde
g_{2d}$ is smooth, and by the equation of $W^\flat$ in its form
\eqref{canform}, $B$ is everywhere tangent to $W^\flat$,
along $(g=\tilde f_k=0)$.

Now the existence of a family of the asserted dimension of such $B$'s
comes from the non-unicity of the expression of $g$ in
\eqref{exprg1}, and more precisely of the linear syzygy between
$f_{k-d}^2$ and $f_{k-d}f_k$.
For all $a_{2d-k} \in H^0(\P^n, \O(2d-k))$ we may also write
\begin{equation*}
  g = f_{k-d}^2 \bigl( g_{2d} + 2a_{2d-k} f_k \bigr)
  +2 f_{k-d}f_k \bigl( g_d - a_{2d-k} f_{k-d} \bigr)
  + f_k^2,
\end{equation*}
and the same computations as above yield that
\(
g = \hat f_k ^2 + f_{k-d}^2 \hat g_{2d}
\),
with
\begin{equation*}
\begin{aligned}
  \hat f_k &= f_k + f_{k-d} \bigl( g_d- a_{2d-k} f_{k-d} \bigr)
  \\
  & = \tilde f_k - a_{2d-k} f_{k-d}^2
\end{aligned}
\qquad \text{and}
\qquad
\begin{aligned}
  \hat g_{2d} &= \bigl( g_{2d} + 2a_{2d-k} f_k \bigr)
  - \bigl( g_d - a_{2d-k} f_{k-d} \bigr)^2
  \\
  & = \tilde g_{2d}
  + a_{2d-k} \bigl(
  2 \tilde f_k- a_{2d-k} f_{k-d} ^2
  \bigr).
\end{aligned}
\end{equation*}
The hypersurfaces defined by $\hat g_{2d}$ move in dimension
$h^0(\P^n, \O(2d-k))$ and have the asserted properties.
\end{proof}

\bigskip
The above Proposition~\ref{pr:everywhere-tg} shows that the forgetful
map $q$ is dominant with fibres of dimension at least
$h^0(\P^n,\O(2d-k))$. The next statement shows that this is indeed the
exact dimension, by a dimension count.

\begin{proposition}
\label{pr:dim}
For all $k \geq d$, we have
\[
\dim(\VW {k,d})-\dim(\W {k,d}) = 
h^0(\P^n,\O(2d-k)).
\]
\end{proposition}

\begin{proof}
The space $\VW {k,d}$ is a projective bundle over $\V d$,
and for $V \in \V d$ we have $\dim(|kL|)=h(k)+h(k-d)-1$ by
\ref{p:double-covers}, therefore
\begin{equation*}
\dim(\VW {k,d})=\dim(\V d)+h(k)+h(k-d)-1
= h(2d)+h(k)+h(k-d)-2.
\end{equation*}
If $k>d$, for all $Z \in \Z {k,d}$, there is a unique degree $k-d$
hypersurface $D$ which contains $Z$. In turn the linear system $|kH|$ on
$D$ has dimension
$h^0(D,kH)-1=h(k)-h(d)-1$,
hence
\begin{equation*}
\dim(\Z {k,d}) = h(k-d)+h(k)-h(d)-2.
\end{equation*}
Eventually for each $Z\in \Z {k,d}$, hypersurfaces of degree $2k$
singular along $Z$ form a linear system of dimension
\begin{align*}
h^0(\P^n,\I_Z^{(2)}(2k))-1
&= h(0)+h(d)+h(2d)-h(2d-k)-h(d-k)-1
\\
&= h(d)+h(2d)-h(2d-k)
\end{align*}
where the first equality is obtained from the resolution of $\I_Z^{(2)}$
given in Lemma~\ref{l:resol}.
We thus find
\begin{equation*}
\dim(\W {k,d}) =
\begin{cases}
h(k-d)+h(k)+h(2d) -h(2d-k)-2
& \text{if } k>d \\
h(2d)-1
& \text{if } k=d.
\end{cases}
\end{equation*}
hence
\begin{equation*}
\dim(\VW {k,d}) - \dim(\W {k,d}) =
\begin{cases}
h(2d-k)
& \text{if } k>d \\
h(d)
& \text{if } k=d.
\end{cases}
\end{equation*}
\end{proof}

\bigskip
We may now prove our main result stated above.

\begin{proof}[Proof of Theorem~\ref{t:main}]
Proposition~\ref{pr:everywhere-tg} shows that $q$ is dominant,
since for general $W^\flat \in \W {k,d}$ and for all $B$ as in
Proposition~\ref{pr:everywhere-tg},
the double cover $\pi: V \to \P^n$ branched over $B$
has two copies of $W$,
the partial normalization $W$ of $W^\flat$
along $Z\in \Z {k,d}$, as membres of $|\pi^*\O(k)|$.
Namely, the pull-back $\pi^* W^\flat$ splits as two copies of $W$,
as explained in Section~\ref{S:double-covers}.

On the other hand,
the two spaces  $\VW {k,d}$ and $\W {k,d}$ are seen to be irreducible
from their descriptions in the proof of Proposition~\ref{pr:dim}, and the
dimension of the generic fibre of $q$ is
$\dim(\VW {k,d})-\dim(\W {k,d}) = h(2d-k)$ by the same
Proposition~\ref{pr:dim}. 
\end{proof}

\bigskip
In fact it is even possible to prove Theorem~\ref{t:main} without the
constructive Proposition~\ref{pr:everywhere-tg}.
By Proposition~\ref{pr:dim} it is enough to show that the generic
fibre of $q$ has dimension $h^0(\P^n, \O(2d-k))$, and this may be seen
as follows.
Let $(V,W) \in \VW {k,d}$. Then $V$ is a degree $2d$ hypersurface in
the weighted projective space $\P(1^{n+1},d)$ and $W$ is cut out on
$V$ by a degree $k$ hypersurface in $\P(1^{n+1},d)$,
see \ref{p:wps}.

If $k>d$,
those $V'$ such that $(V',W)$ belongs to $\VW {k,d}$ form the linear
system
\[
  \P H^0 \bigl(
  \P(1^{n+1},d), \I_W(2d)
  \bigr),
\]
which has dimension
\[
  h^0 \bigl(
  \P(1^{n+1},d), \O(2d-k)
  \bigr)
  =
  h^0 \bigl( \P^n, \O(2d-k) \bigr),
\]
the equality coming from the fact that $2d-k<d$, hence degree $2d-k$
polynomials don't involve the weight $d$ variable.

If $k \leq d$ one has to take care of the fact that there are
automorphisms of $\P(1^{n+1},d)$ fixing $W$, and degree $k$
hypersurfaces in the same orbit under the action of automorphisms of
$\P(1^{n+1},d)$ give the same double cover.
When $k=d$ the automorphisms fixing $W$ are all of the form
$(\bx,y) \mapsto (\bx,ay)$ for some $a \in \CC$ hence form a
$1$-dimensional group, whereas the linear system of $(2k)$-ics
containing $W$ has dimension
\[
  h^0 \bigl(
  \P(1^{n+1},d), \O(d)
  \bigr)
  =
  h^0 \bigl( \P^n, \O(d) \bigr) +1.
\]

\section{Curves on $K3$ surfaces of genus $2$}
\label{S:K3-g2}

For all integers $k\geq 3$, we let $\C k$ be the locus in 
$\M_{k^2+1}$ of those curves $C$ that have a plane model $C^\flat$ of
degree $2k$ with $k(k-3)$ nodes forming together
a complete intersection of type $(k,k-3)$, and no further
singularity. 

\begin{theorem}
\label{t:g2}
For all $k\geq 3$,
the forgetful map
$c_{k^2+1}^k: \KC_{k^2+1}^k \to \M_{k^2+1}$ dominates $\C k$.\footnote
{by this we mean by a slight abuse of terminology that the Zariski
closure of the image of $c_{k^2+1}^k$ coincides with that of $\C k$.}
For very general $C \in \C k$, 
the Gauss--Wahl map of $C$ has corank
\[
\cork(\Phi_C)
= h^0(\P^2, \O(6-k)) +1 -\nu_2(C),
\]
with
\[
\nu_2(C) =  h^0\bigl(
N_{C/\P^{g-1}}(-2)
\bigr)
= 
\begin{cases}
1 & \text{if}\ k=3 \\
0 & \text{if}\ k \geq 4.
\end{cases}
\]
\end{theorem}

\noindent
(The normal bundle $N_{C/\P^{g-1}}$ is with respect to the canonical
embedding $C \subset \P^{g-1}$).

\begin{remark}
For $k=1,2$, the image of $c_{k^2+1}^k$ is the hyperelliptic locus in
$\M_{k^2+1}$, and the generic fibre has dimension $18$ and $15$
respectively. It is not clear in these cases how to relate the 
dimension of the generic fibre to the corank of the Gauss--Wahl map,
which is $3g-2$ for all genus $g$ hyperelliptic curves
\cite{wahl90, cm92}.
\end{remark}

\begin{proof}[Proof of Theorem~\ref{t:g2}]%
We consider the family $K_2$ of all smooth plane sextics;
it is a dense open subset of $\V 3$, in the notation of 
\ref{p:notation:VW}.
For all positive integers $k$, the moduli space $\K_{k^2+1}^k$ is
isomorphic to the quotient of $K_2$ by the action of $\PGL(3)$.
Over $K_2$, we consider for all $k\geq 3$ the dense open subset
$\sKC_{k^2+1}^k$ of $\VW {k,3}$ parametrizing pairs $(S,C)$ with $S$ a
smooth sextic double plane and $C$ a smooth member of the linear
system $|kL|$ on $S$, in our usual notation.
Then by Theorem~\ref{t:main} the forgetful map
$\sKC_{k^2+1}^k \to \W {k,3}$ is dominant, with general fibre of
dimension $h(6-k)$.
After we divide out by the action of $\PGL(3)$, we obtain that 
$c_{k^2+1}^k$ dominates $\C k$, with general fiber of dimension
$h(6-k)$. 

For very general $C \in \C k$ there exists, as we now know, a
primitive $K3$ surface $(S,L)$ of genus $2$ with $\Pic(S)=\Zint.L$,
such that $C \in |kL|$.
By \cite{green-lazarsfeld} the Clifford index of $C$ is computed by a
line bundle on $S$, hence
\[
  \Cliff(C) = \deg(\restr L C)-2\bigl(
  h^0(\restr L C) -1 \bigr) = 2k-4.
\]
For $k>3$ we thus have $\Cliff(C)>2$,
which implies that $\nu_2(C)=0$ 
(see \cite[{\SS 3}]{cds} and the references therein).
Moreover $C$ has genus $k^2+1 \geq 11$, so we
have by \cite[Thm.~2.6]{cds} that
\[
\cork(\Phi_C) = 
\dim \bigl( (c_{k^2+1}^k)^{-1}(C) \bigr) +1,
\]
which completes the proof in this case.

When $k=3$, $C$ is a smooth plane sextic so it has $\nu_2(C)=1$
(see \cite[Prop.4.3]{knutsen18}; 
this may also be proved as in 
Proposition~\ref{pr:cohom}),
and $\cork(\Phi_C)=10$ as for all smooth plane curves of degree $d\geq 5$
\cite[Rmk.~4.9]{wahl90}.
\end{proof}

\begin{remark}
\label{rem:rat-ell-exts}
Let $\Gamma$ be a nodal plane curve whose set of nodes $Z$ lies on
a cubic curve $T$. Then the rational surface $S_T$ obtained by
blowing-up the plane at the points $\Gamma \cap T$ provides an extension of
the canonical model of the normalization of $\Gamma$, see \cite[\SS
9]{cds} and the references therein.

Now let $k\geq 3$ and $C$ be a general member of $\C k$.
Then $C^\flat$ is a nodal curve with set of nodes $Z$, and there exists a
linear system of dimension $h(6-k)-1$ of cubics containing $Z$.
These cubics in turn give a family of the same dimension of mutually
non-isomorphic rational surfaces extending the canonical model of $C$.

If $k>3$, then $C$ has Clifford index $>2$ and genus $\geq 11$, so
that by \cite{cds} its canonical model has a universal
extension. Then the existence of a $K3$ extension of $C$, together
with the aforementioned rational extensions imply that the family of
extensions of $C$ has dimension $\geq h(6-k)$ hence
$\cork(\Phi_C) \geq h(6-k)+1$.
Since the latter inequality is in fact an equality, we see that those
rational extensions of $C$ form a divisor in the universal family of
all extensions of $C$.

Moreover note that $C$ may be seen as a smooth curve in the surface
$S'$ obtained by blowing up $\P^2$ at $Z$. This surface has
$h^1(S',\O_{S'})=0$ and $h^0(S',-K_{S'})=h(6-k)$, so that
\[
  \cork(\Phi_C)
  > h^0(S',-K_{S'}),
\]
thus showing that in \cite[Conjecture p.80]{wahl90} the inequality may
be strict.
\end{remark}

\paragraph{}
Prokhorov \cite{prokhorov} has proved that all Fano threefolds $V$ with
canonical Gorenstein singularities have $-K_V^3 \leq 72$, and that
there are only two such threefolds reaching this bound, namely the two
weighted projective spaces
\[
Y_1=\P(3,1,1,1)
\quad \text{and}
\quad
Y_2=\P(6,4,1,1).
\]
This implies that no $K3$ surface of genus $g>37$ is extendable
\cite[Cor.~1.6]{prokhorov}.

The threefold $Y_1$ has Picard group generated by $\O(1)$, and the
linear system $|\O(3)|$ embeds it in $\P^{10}$ as the cone over the
$3$-Veronese surface $v_3(\P^2) \subset \P^9$. One has $\omega_V \cong
\O(-6)$, so the anti-canonical model of $Y_1$ is the $2$-Veronese
re-immersion of $Y_1 \subset \P^{10}$ in $\P^{38}$, and the general $S
\in |-K_{Y_1}|$ is a $K3$ double cover of $\P^2$ branched over a sextic.
In turn, hyperplane sections of $S \subset \P^{37}$ are members of the
linear system $|6L|$, in our usual notation.

\begin{corollary}
\label{c:non-ext-Y1}
The anti-canonical embedding of
$Y_1=\P(3,1,1,1)$
is not extendable.
\end{corollary}

\begin{proof}
Let $C$ be a smooth curve linear section of $Y_1$. It is
non-degenerate in $\P^{36}$, and has
$\alpha(C)=2<36$ by Theorem~\ref{t:g2}. 
It thus follows from \cite[Thm.~0.1]{lvovskiII} that $C$ cannot be
extended more than $2$ steps, hence $Y_1$ cannot be extended.
\end{proof}

\begin{remark}
\label{r:ext-ell-sing}
Let $C \in \C 6$ be general. Corollary~\ref{c:non-ext-Y1} tells us in
particular 
that $Y_1$ is the universal extension of the canonical model of $C$.
Continuing Remark~\ref{rem:rat-ell-exts}, we point out that
there is a unique plane cubic $T$ passing through the 18 nodes of
$C^\flat$, and the corresponding rational extension $S_T$ is the unique
hyperplane section of $Y_1$ containing $C$ and passing through the
unique singular point of $Y_1$ (its vertex in the model as the cone
over $v_3(\P^2)$).
Eventually, note that $C\cap T$ set-theoretically consists of $18$
points, which is the maximum for $S_T$ to be smoothable to a $K3$
surface (see \cite{abs1} and the references therein for more details). 
\end{remark}

\paragraph{}
The threefold $Y_2$ on the other hand has Picard group generated by
$\O(12)$. It is singular along the line joining $P_0=(1:0:0:0)$ 
and $P_1=(0:1:0:0)$; along this line and off $P_0$, $Y_2$ looks
analytically locally
like a line times a surface
ordinary double point. The general $S \in |-K_{Y_2}|$ has an ordinary
double point at its intersection with the line $P_0P_1$, and has
Picard group generated by $\O(2)$.
Thus $S$ is a $K3$ surface of genus $37$ with $6$-divisible hyperplane
sections, as
are the anti-canonical sections of $Y_1$.
Still our results don't let us tell whether $Y_2$ is extendable or
not.\footnote
{After this article has been written, it has been proved
\cite[Cor.~6.4]{wps} that $Y_2$ is not extendable.}

\paragraph{}
\label{p:totaro}
For $k=4$ and $5$ respectively,
Theorem~\ref{t:g2} together with \cite[Thm.~2.1]{cds}
indicates the existence of arithmetically Gorenstein Fano
varieties of 
dimensions $8$ and $5$,
indices $6$ and $3$,
and genera $17$ and $26$ respectively.
These have been identified by Totaro (personal communication).
Both are sextic hypersurfaces in a weighted projective space.
The one for $k=4$ is defined by
\[
  0 = y^2 + x_0^2z_0 + x_0x_1z_1 + x_0x_2z_2
      + x_1^2z_3 + x_1x_2z_4 + x_2^2z_5,
\]
in $\P(4^6, 3, 1^3)$,
where the $x_i$'s have weight $1$, $y$ has weight $3$, and
the $z_i$'s have weight $4$;
the one for $k=5$ is defined by
\[
  0 = y^2 + x_0z_0 + x_1z_1 + x_2z_2,
\]
in $\P(5^3, 3, 1^3)$, 
where the $x_i$'s have weight $1$, $y$ has weight $3$, and
the $z_i$'s have weight $5$.

These extensions and similar universal extensions for complete
intersection $K3$ surfaces are described in detail in
\cite[\SS 3]{cd-highind}, see also
\cite{th-5iche}.

\section{Miscellaneous remarks on plane curves}
\label{S:severi}

\paragraph{Superabundant log-Severi varieties}
Fix a curve $B$ of degree $2d$, and let $\pi: S \to \P^2$ be the
double cover ramified over $B$. 
We consider the image 
$V_{B,k}$ in $|2kH|$ of the linear system
$|kL|$ on $S$. It has dimension
\[
h^0(S,kL)-1
= h(k) + h(k-d) -1,
\]
and parametrizes curves of geometric genus
\[
g_{k,d} 
=\textstyle \frac 1 2 (2k-1)(2k-2)-k(k-d)
\]
everywhere tangent to $B$, the number of contact points is thus
$2kd$.

The family of curves $V_{B,k}$ is therefore contained in the
log-Severi variety
$V _{g_{k,d}} ^{2kH}(0,[0,2kd])$ of the pair $(\P^2,B)$, 
in the notation of \cite{thomasCH}.
By definition, this log-Severi variety parametrizes plane curves of
degree $2k$ and geometric genus $g_{k,d}$, tangent to $B$ at $2kd$
unassigned points. 
The expected dimension of this log-Severi variety is
\begin{align*}
-(K_{\P^2}+B)\cdot 2kH+g_{k,d}-1+2kd
&= k(k+3-d),
\end{align*}
and by \cite[(1.4.0)]{thomasCH} a component of the Severi variety has
the expected dimension if 
it has an irreducible member and
\[
-K_{\P^2}\cdot 2kH - 2kd \geq 1
\quad \iff
\quad
2k(3-d) \geq 1.
\]
The latter inequality holds if and only if $d \leq 2$.

It turns out that the dimension of our family $V_{B,k}$ exceeds the
expected dimension of the log-Severi variety.
Indeed a direct computation shows that
\begin{align*}
\dim(V_{B,k})
- \expdim
\bigl(V _{g_{k,d}} ^{2kH}(0,[0,2kd])\bigr)
&=  \textstyle
\frac {(d-1)(d-2)} 2
\\
&= p_g(S)
\end{align*}
(\cf \cite[V.22 p.237]{bhpv} for the last equality).

\bigskip
The two following examples are applications of our main result, 
Theorem~\ref{t:main}, in the cases $n=2$ and $d=1,2$
respectively. 

\paragraph{Sections of quadric surfaces} 
A quadric in $\P^3$ is a double cover
of the plane branched over a conic. In this case we obtain 
for all $k\geq 1$ a correspondence between 
\\ a) smooth complete intersections of bi-degree $(2,k)$ in $\P^3$,
and
\\ b) plane curves of degree $2k$ with $k(k-1)$ nodes at a complete
intersection of bidegree $(k,k-1)$.

We remark that these curves are $k$-gonal, which is not immediately
apparent from the presentation b).


\paragraph{Sections of degree $2$ Del Pezzo surfaces}
A Del Pezzo surface of degree $2$ is a double cover 
$\pi:S \to \P^2$
branched over a quartic. It is also isomorphic to the blow-up of
$\P^2$ at seven points in general position, and one has
$\pi^* H \simeq -K_S$.
We thus obtain for all $k \geq 2$ a correspondence between 
\\ a) plane curves of degree $3k$ with seven $k$-fold points in general
position, and
\\ b) plane curves of degree $2k$ with $k(k-2)$ nodes at a complete
intersection of bi-degree $(k,k-2)$.

We observe that the above curves are Cremona-minimal, \ie
it is not possible to lower their degrees by applying Cremona
transformations.
It thus happens that a curve of type a) is birational to a curve of
type b), but the birational isomorphism between them may not be
realized by a Cremona transformation of the projective plane.

\section{$K3$ surfaces in Fano double projective spaces}
\label{S:double-Fano}

In this section we apply our main result Theorem~\ref{t:main} to
study the extensions of those $K3$ surfaces that are anti-canonical
divisors in Fano solids double covers of the projective space
$\P^3$. 

Let $\pi: V \to\P^3$ be a double cover branched over a hypersurface
$B$ of degree $2d$. 
We have
$K_V=(d-4)L$, hence $V$ is Fano if and only if $d<4$.
Theorem~\ref{t:main} is relevant only for sections of $|kL|$
with $k\geq d$, which leaves only the two cases $d=1,2$ since we 
take $k=4-d$.
For $d=3$, members of $|-K_V|$ are sextic double planes and these have
already been considered in Section~\ref{S:K3-g2}.

\paragraph{}
When $d=1$, $V$ is a double cover of $\P^3$ branched over a quadric, 
and this is a quadric hypersurface in $\P^4$.
Then $K_V=-3L$ and anticanonical divisors of $V$ are complete
intersections of type $(2,3)$ in $\P^4$, which is the general form of
genus $4$ $K3$ surfaces.

The general $S \in |-K_V|$ is mapped birationally by $\pi$ to a
sextic surface, double along a complete intersection of type
$(3,2)$. 
Conversely, every such sextic may be unprojected to a cubic section of
a quadric in $\P^4$, in a unique way by Theorem~\ref{t:main}.

\paragraph{}
The case $d=2$ is more interesting.
Then $V$ is a double cover of $\P^3$ branched over a quartic surface,
$K_V=-2L$, and
the general $S \in |-K_V|$
is mapped birationally by $\pi$ to a smooth quartic surface.
In its anticanonical embedding $V$ is a quadric section of the cone in
$\P^{10}$ over the Veronese variety $v_2(\P^3) \subset
\P^9$. Hyperplane sections of $S$ in its embedding in $\P^9$ are
$2$-Veronese re-immersions of quadric sections of $S$ in $\P^3$.
In this case we obtain the following results.

\begin{theorem}
\label{t:double-fano}
For general $C$ in the image of $c^2_9$,
the fibre of $c_9 ^2$ over $C$ has
dimension $\cork(\Phi_C)-1+\nu_2(C)$
with
\[
\cork(\Phi_C)=10
\quad \text{and}
\quad
\nu_2(C)=h^0(N_{C/\P^8}(-2))=1.
\]
For $(S,C)$ in this fibre, the family of pairs $(S',C)$ such that the
ribbons $2C_S$ and $2C_{S'}$ are isomorphic has dimension $\nu_2(C)$.

Moreover, the fibre of $s_9^2: \FS_9^2 \to \K_9^2$ over $S$ has
dimension $h^1(T_S(-1))-1+\nu_2(S)$
with
\[
h^1(T_S(-1)) = 10
\quad \text{and}
\quad
\nu_2(S)=h^0(N_{S/\P^9}(-2))=1.
\]
There is one $(V,S)$ in this fibre such that the ribbon $2S_V$ is
trivial. 
\end{theorem}

\begin{proof}[Proof of Theorem~\ref{t:double-fano}]%
Let $C$ be a general member of the image of $c_9^2$.
There exist a quadric $Q$ and a quartic $S$ in
$\P^3$ such that $C=Q\cap S$.
One has $\omega_C = \restr {\O_{\P^3}(2)} C$,
and the canonical model of $C$ is its $2$-Veronese re-immersion in
$\P^8$. 
General $K3$ extensions of the canonical model of $C$ are $2$-Veronese
re-immersion in $\P^8$ of smooth quartics in $\P^3$.

Let $1C=C$ and $2C$ be the ribbon of $C$ in $S$. The latter ribbon is cut
out on $S$ by the square of an equation of $Q$. We thus have that for
all $k=1,2$, $kC$ is a complete intersection of type $(2k,4)$ in
$\P^3$. It follows that
\[
h^0(\P^3,\I_{kC}(4))-1 = h(0)+h(4-2k)-1
=
\begin{cases}
10 & \text{if } k=1 \\
1 & \text{if } k=2,
\end{cases}
\]
which proves that extensions of $C$ form a $10$-dimensional family, in
which those $S'$ such that $2C_{S'}=2C$ form a $1$-dimensional
subfamily. 

In turn the general extension of $S \subset \P^9$ is the
anti-canonical model of a quartic double $\P^3$. 
These form a $10$-dimensional family by Theorem~\ref{t:main}. 
A particular member is the double cover branched over $S$ itself, in
which the ribbon of $S$ is trivial, see
\cite[(2) p.875]{beauville-merindol}.

It only remains to prove the assertions about 
$\alpha(X)$ and $\nu_2(X)$ for $X=C,S$. This is the object of
Proposition~\ref{pr:cohom} (note that as $V \subset \P^{10}$ is a
quadric section of the cone over $v_2(\P^3)$, its hyperplane section
$S\subset \P^9$ is a quadric section of $v_2(\P^3)$).
\end{proof}

\begin{proposition}
\label{pr:cohom}
Let $Y =v_2(\P^3) \subset \P^9$ be the $2$-Veronese embedding of
$\P^3$, $S$ a smooth quadric section of $Y$, and $C$ a smooth
hyperplane section of $S$. We have:\\
\begin{inparaenum}[{\normalfont (\roman{enumi})}]
\item $h^0(N_{Y/\P^9}(-1))=10$ and
$h^0(N_{Y/\P^9}(-k))=0$ for all $k\geq 2$;
\\ \item $h^0(N_{S/\P^9}(-1))=10+10$,
$h^0(N_{S/\P^9}(-2))=1$,
and $h^0(N_{S/\P^9}(-k))=0$ for all $k\geq 3$;
\\ \item $h^0(N_{C/\P^8}(-1))=10+9$,
$h^0(N_{C/\P^8}(-2))=1$,
and $h^0(N_{C/\P^8}(-k))=0$ for all $k\geq 3$.
\end{inparaenum}
\end{proposition}

\begin{proof}
We follow \cite{alzati-re} to compute the cohomology of 
$N_{Y/\P^9}$.
Let $E$ be the vector space underlying $\P^3$. 
For all $k$ we have an exact sequence 
\begin{equation*}
0 \to
E \otimes \O_{\P^3} (1-2k) 
\to S^2E \otimes \O_{\P^3} (2-2k)
\to N _{Y/\P^9}(-k)
\to 0
\end{equation*}
(beware that in \cite{alzati-re}
$N_Y(1)$ means $N_Y \otimes \O_{\P^3}(1)$, 
whereas we take it as
$N_Y \otimes \O_{\P^9}(1)= N_Y \otimes \O_{\P^3}(2)$).
This gives the dimensions in (i), as well as the vanishing
of $H^1 (N_Y(-k))$ for all $k \in \ZZ$,
and of $H^2 (N_Y(-k))$ for all $k \leq 2$.
Moreover, by \cite[Thm.~2]{alzati-re}
$H^2 (N_Y(-k))$ also vanishes for all $k\geq 4$,
and has dimension $6$ for $k=3$
(to apply \cite[Thm.~2]{alzati-re} to our situation, one should
substitute their $U$, $T$, and $k$ to our $E$, $\{0\}$, and $2k$
respectively, so that 
for us their $S^\chi \otimes T$ is always $\{0\}$ and
$\mu^{-1}(T)$ is the kernel of the multiplication map
$E \otimes E \to S^2E$;
also note that there is a misprint in the statement of
\cite[Thm.~2]{alzati-re}, as the assumption ``$k \geq d-n-1$'' should
read ``$k \geq d+n+1$'').

We start by computing the cohomology of $N_{S/\P^9}$.
We shall use the exact sequence 
\begin{equation}
\label{normalS}\tag{$\star$}
0 \to 
N_{S/Y}=\O_S(2) \to
N_{S/\P^9} \to
\restr {N_{Y/\P^9}} S
\to 0.
\end{equation}
To compute the cohomology of $\restr {N_{Y/\P^9}} S$ we use the
restriction exact sequence
\begin{equation*}
0 \to 
N_{Y/\P^9}(-2)
\to N_{Y/\P^9}
\to \restr {N_{Y/\P^9}} S
\to 0.
\end{equation*}
Since $H^1(N_{Y/\P^9}(-k))=0$ for all $k$,
we find that
\[
  h^0(\restr {N_{Y/\P^9}(-k)} S)
  = h^0(N_{Y/\P^9}(-k))
  - h^0(N_{Y/\P^9}(-k-2))
  = \begin{cases}
    10 & \text{if } k=1 \\
    0 & \text{if } k\geq 2.
  \end{cases}
\]
Moreover,
since $H^1(N_{Y/\P^9}(-k))=0$ for all $k$ and
$H^2(N_{Y/\P^9}(-k-2))=0$ for all $k \neq 1$, as we have seen in the
first paragraph of the proof, we find that
$H^1(\restr {N_{Y/\P^9}(-k)} S) =0$ for all $k\neq 1$.

We may now return to the exact sequence \eqref{normalS}. Using the
fact $H^1(\O_S(k))$ vanishes for all $k$, we find that
\[
  h^0(N_{S/\P^9}(-k))
  = h^0(\restr {N_{Y/\P^9}(-k)} S)
  + h^0(\O_S(-k+2)),
\]
which gives the dimensions in part (ii) of the Proposition.
Using in addition the vanishing of
$H^1(\restr {N_{Y/\P^9}(-k)} S) =0$ for $k\neq 1$, we find that
$H^1(N_{S/\P^9}(-k))$ vanishes for all $k \neq 1$.

Next we compute the cohomology of $N_{C/\P^8}$ in a similar
fashion. We consider the exact sequence 
\begin{equation*}
0 \to 
N_{C/S} = \O_C(1) \to
N_{C/\P^9} = N_{C/\P^8} \oplus \O_C(1) \to
\restr {N_{S/\P^9}} C
\to 0
\end{equation*}
which gives an isomorphism
$N_{C/\P^8}
\cong \restr {N_{S/\P^9}} C$.
Then we consider the restriction exact sequence
\begin{equation*}
0 \to 
N_{S/\P^9}(-1)
\to N_{S/\P^9}
\to \restr {N_{S/\P^9}} C = N_{C/\P^8}
\to 0.
\end{equation*}
Using the vanishing of
$H^1(N_{S/\P^9}(-k))$ for all $k \neq 1$, we obtain
for all $k \geq 1$
\[
  h^0(N_{C/\P^8}(-k))
  = h^0 (N_{S/\P^9}(-k)) - h^0 (N_{S/\P^9}(-k-1)),
\]
which gives the dimensions in part (iii) of the
Proposition.
\end{proof}

\paragraph{Remark}
It has been proved by Wahl \cite[Thm.~4.8]{wahl90} that a curve of
genus $g\geq 5$ sitting on a 
rational ruled surface has $\cork (\Phi_C) \geq 9$.
This corresponds in general to the existence of rational surface
extensions, see \cite[\SS 9]{cds}. For curves $C$ as in
Theorem~\ref{t:double-fano}, there are in addition $K3$ extensions.
This would imply that $\cork(\Phi_C)>9$ as in Remark~\ref{r:ext-ell-sing},
should the integral of a ribbon be unique.
This is not the case because $\nu_2(C)=1$, but the corank
nevertheless jumps, as we have observed above.

On the other hand, it has been proved by Brawner \cite{brawner} that a
general tetragonal curve of 
genus $g \geq 7$ has Wahl map of corank $9$. 
Curves as in Theorem~\ref{t:double-fano} have two $g^1_4$'s, and thus
provide examples of tetragonal curves with Wahl map of exceptional
corank.

\begin{closing}
\providecommand{\bysame}{\leavevmode\hbox to3em{\hrulefill}\thinspace}
\providecommand{\og}{``}
\providecommand{\fg}{''}
\providecommand{\smfandname}{and}
\providecommand{\smfedsname}{eds.}
\providecommand{\smfedname}{ed.}
\providecommand{\smfmastersthesisname}{M{\'e}moire}
\providecommand{\smfphdthesisname}{Th{\`e}se}

\medskip\noindent
Ciro Ciliberto.
Dipartimento di Matematica.
Universit{\`a} degli Studi di Roma Tor Vergata.
Via della Ricerca Scientifica,
00133 Roma, Italy.
\texttt{cilibert@mat.uniroma2.it}

\medskip\noindent
Thomas Dedieu.
Institut de Math{\'e}matiques de Toulouse~; UMR5219.
Universit{\'e} de Toulouse~; CNRS.
UPS IMT, F-31062 Toulouse Cedex 9, France.
\texttt{thomas.dedieu@math.univ-toulouse.fr}

\renewcommand{\thefootnote}{}
\footnotetext
{CC is a member of GNSAGA of INdAM. He
acknowledges the MIUR Excellence Department Project awarded to the
Department of Mathematics, University of Rome Tor Vergata, CUP
E83C18000100006.}
\end{closing}

\end{document}